\documentclass[11pt,letterpaper]{article}

\usepackage[margin=1in]{geometry}
\usepackage[default]{opensans}
\usepackage{amsmath,amssymb,amsthm}
\usepackage{graphicx}
\usepackage[svgnames,dvipsnames]{xcolor}
\usepackage{hyperref}
\usepackage{natbib}
\usepackage{caption}

\hypersetup{colorlinks=true, linkcolor=blue!50!black, citecolor=teal!60!black, urlcolor=blue!60!black}
\graphicspath{{figures/}{letter/figures/}}

\newtheorem{theorem}{Theorem}
\newtheorem{proposition}[theorem]{Proposition}
\theoremstyle{definition}
\newtheorem{remark}[theorem]{Remark}
\newtheorem{assumption}[theorem]{Assumption}

\newcommand{\IR}{\mathrm{I}\!\mathrm{R}}
\newcommand{\bigO}{\mathcal{O}}
\newcommand{\eps}{\varepsilon}

\title{\textbf{Rockafellian relaxation and minimum-norm slack for the Walrasian equilibrium problem}}
\author{Julio Deride\thanks{Universidad Adolfo Ib\'a\~nez, Pe\~nalol\'en, Santiago, Chile.
		\texttt{julio.deride@uai.cl}}}
\date{\today}

\begin{document}
	
	\maketitle
	
	\begin{abstract}\noindent
		We propose a Rockafellian relaxation of the Walrasian
		equilibrium problem for an exchange economy that may not admit
		one. Market clearing is slackened by a non-negative variable
		$v$ whose norm is penalized; the relaxation is well posed
		throughout. As the penalty grows, the residual converges to a
		vector $v^*_\infty$ of minimum norm in the feasible range of
		excess demand, measuring the distance to the nearest
		equilibrium-admitting economy. A stressed Shapley--Shubik
		example recovers the analytical infeasibility floor to machine
	\end{abstract}
	
	\noindent\textbf{Keywords:} Rockafellian relaxation, Walrasian equilibrium,
	non-existence, augmented Walrasian, constraint perturbation,
	quasi-linear utilities.\\
	\textbf{Mathematics Subject Classification:} 90C30, 91B50, 49J53, 90C46.
	
	\bigskip
	
	\section{Introduction}\label{sec:intro}
	
	In this paper, we propose a Rockafellian relaxation of the
	Walrasian equilibrium problem for an exchange economy and study its
	behavior when no equilibrium exists. The relaxation encodes
	market clearing through a non-negative slack variable
	$v \in \IR^L_+$ and penalizes its norm. The inner minimization in
	$v$ admits a closed form, so the relaxation reduces to a single
	regularized allocation problem in $x$, which we solve at increasing
	values of a penalty parameter $\lambda$. As $\lambda \to \infty$,
	we show that the residual $v^*(\lambda)$ converges to a vector
	$v^*_\infty$ of minimum norm in the feasible range of the
	excess-demand map (Theorem~\ref{thm:nonex}). When an equilibrium
	exists, $v^*_\infty = 0$ and we recover a planner-side
	characterization of equilibrium; when it does not,
	$\|v^*_\infty\|$ measures the distance from the economy to the
	nearest equilibrium-admitting one, and the choice of norm in the
	penalty fixes the policy reading (Remark~\ref{rem:norms}).
	
	Familiar sources of equilibrium non-existence---endowments
	violating survival, short-sale constraints binding tightly in
	incomplete asset markets, non-convex preferences---produce
	computational symptoms that classical schemes cannot interpret.
	The homotopy and fixed-point algorithms of \citet{ScrH73, Eves72,
	Sgal83} and the projection schemes surveyed in \citet{Judd98,
	BrKu08} are designed under the implicit assumption that an
	equilibrium exists; when it fails to, iterates diverge, stagnate,
	or terminate at numerical artifacts that lack economic
	interpretation. We show that the Rockafellian relaxation we
	propose degrades gracefully in such situations and returns a
	minimum-norm slack that admits a direct economic reading.
	
	The construction sits in the family of \emph{centralized}
	reformulations of equilibrium initiated by \citet{Negishi60} and
	underlying the welfare theorems, where competitive equilibrium is
	recovered as a planner optimum. We follow the perturbation
	framework of \citet[Ch.~11]{RocW98} and the \emph{Rockafellian}
	viewpoint articulated in \citet{Royset21}, with the approximation
	and stability theory developed in our companion paper
	\citet{DerR25}. The \emph{augmented Walrasian} iteration of
	\citet{DerJW19, DerJWXR} stabilizes the dual side of the
	equilibrium problem by a similar augmentation device, but does not
	yield a residual diagnostic when equilibrium fails; that is the
	gap this note fills.
	
	The article is organized as follows. Section~\ref{sec:method}
	defines the primal Rockafellian, derives the closed-form reduction
	of the relaxation, and proves the minimum-norm characterization of
	the limit residual. Section~\ref{sec:numerics} illustrates the
	construction on a stressed variant of the Shapley--Shubik
	$2\times 2$ economy of \citet{ShapShub77, BergSY09}.
	Section~\ref{sec:concl} concludes.
	
	\section{The Rockafellian relaxation}\label{sec:method}
	
	We consider an exchange economy $\mathcal{E}$ with $L$ goods and $I$
	agents, indexed by $i = 1, \dots, I$. Each agent has an endowment
	$e^i \in \IR^L_+$, a survival set $X_i \subset \IR^L$, and a concave,
	upper semicontinuous utility $u_i$. We write $X = \prod_i X_i$,
	$U(x) = \sum_i u_i(x_i)$ for the aggregate utility, and
	\[
	Z(x) := \sum_{i=1}^I (x_i - e^i)
	\]
	for the aggregate excess demand at allocation $x$. A Walrasian
	equilibrium is a price vector $p^* \in \IR^L_+$ together with an
	allocation $x^*$ such that each $x^*_i$ maximizes $u_i$ subject to
	$\langle p^*, x_i \rangle \le \langle p^*, e^i \rangle$, with
	$Z(x^*) \le 0$ and $\langle p^*, Z(x^*) \rangle = 0$.
	
	\subsection{The primal Rockafellian and its relaxation}
	
	We follow the perturbation framework of \citet[Ch.~11]{RocW98}
	and \citet{Royset21} and encode market clearing through a
	non-negative slack variable $v \in \IR^L_+$. Define the primal
	Rockafellian
	\begin{equation}\label{eq:F}
		F(x, v) := \begin{cases} -U(x) & \text{if } x \in X,\ v \in \IR^L_+,\ \text{and } Z(x) \le v, \\ +\infty & \text{otherwise.}\end{cases}
	\end{equation}
	Each component $v_\ell$ measures supply slack in good $\ell$. The
	nominal problem $\min_x F(x, 0)$ is the welfare-maximizing planner
	problem under market clearing, so \eqref{eq:F} sits in the family
	of centralized reformulations of the equilibrium problem alongside
	\citet{Negishi60}.
	
	Although $F$ takes the form of a classical perturbation function
	in the sense of \citet{Rocf74}, we follow \citet{Royset21} and
	\citet{DerR25} in calling it a \emph{Rockafellian} to emphasize
	that our interest is in the relaxation~\eqref{eq:relax} as a
	numerically stable substitute problem and in the
	variational-analytic approximation theory it admits, rather than
	in the conjugate dual of the nominal problem.
	
	\begin{proposition}[Centralized characterization, quasi-linear case]\label{prop:planner}
		Suppose $X$ is compact, $U$ is continuous on $X$, and each $u_i$
		is quasi-linear in a common num\'eraire good. If $(x^*, p^*)$ is
		a Walrasian equilibrium of $\mathcal{E}$ and at least one good is
		desired by every agent, then $x^*$ minimizes $F(\cdot, 0)$ on
		$X$, with $p^*$ a Lagrange multiplier for the constraint
		$Z(x) \le 0$.
	\end{proposition}
	
	The statement is the First Welfare Theorem in the quasi-linear
	setting of the example below; for general utilities the analogue
	holds with $U$ replaced by the Negishi-weighted aggregate
	$U_\alpha(x) = \sum_i \alpha_i u_i(x_i)$, with weights $\alpha_i$
	proportional to the inverses of the agents' marginal utilities of
	income at equilibrium~\citep{Negishi60}. The relaxation analysis
	below applies verbatim to either form.
	
	We define the $\lambda$-Rockafellian relaxation of $F$ as
	\begin{equation}\label{eq:relax}
		\min_{x \in \IR^L,\, v \in \IR^L} \; F(x, v) + \tfrac{\lambda}{2}\|v\|_2^2
		\;=\; \min_{x \in X} \;\; -U(x) + \tfrac{\lambda}{2}\|Z(x)_+\|_2^2,
	\end{equation}
	where $Z(x)_+$ denotes the componentwise positive part. The
	reduction on the right follows from the inner minimization in
	$v$, attained at $v^* = Z(x)_+$. Notice that, when an equilibrium
	$x^*$ exists, $\|Z(x^*)_+\|_2 = 0$ and the relaxation recovers
	Proposition~\ref{prop:planner} in the limit $\lambda \to \infty$.
	
	\subsection{Non-existence: the minimum-norm residual}
	
	\begin{assumption}[Standing]\label{ass:standing}
		The set $X$ is non-empty and compact, $U$ is continuous on $X$,
		and $Z$ is an affine mapping.
	\end{assumption}
	
	When no $x \in X$ satisfies $Z(x) \le 0$, no Walrasian
	equilibrium exists. We show that the relaxation~\eqref{eq:relax}
	nonetheless remains well posed, and that its solutions converge
	to the minimum-norm slack we now describe. Let
	$r := \min\bigl\{\|Z(x)_+\|_2 : x \in X\bigr\}$, which is attained
	by continuity of $Z(\cdot)_+$ and compactness of $X$.
	
	\begin{theorem}[Minimum-norm residual, $\ell_2$ form]\label{thm:nonex}
		Under Assumption~\ref{ass:standing}, suppose the feasibility set
		$\{x \in X : Z(x) \le 0\}$ is empty. Let $x^*(\lambda) \in X$ be
		any $1/\lambda$-approximate minimizer of
		$-U(x) + (\lambda/2)\|Z(x)_+\|_2^2$ over $X$. Then the residual
		$v^*(\lambda) := Z(x^*(\lambda))_+$ has a subsequential limit $v^*_\infty$
		that satisfies
		\begin{equation}\label{eq:minnorm}
			\|v^*_\infty\|_2 = r \;=\; \min\bigl\{\|v\|_2 \,:\, v \in \IR^L_+,\; \exists\, x \in X \text{ with } Z(x) \le v\bigr\}.
		\end{equation}
	\end{theorem}
	
	\begin{proof}
		Let $U_{\sup} := \sup_{x \in X} U(x) < \infty$ by
		Assumption~\ref{ass:standing}, and write
		$\phi_\lambda(x) := -U(x) + (\lambda/2)\|Z(x)_+\|_2^2$ for the
		objective. Since $\phi_\lambda \ge -U_{\sup}$ on $X$, the infimum
		$\phi^*_\lambda := \inf_{x \in X} \phi_\lambda(x)$ is finite. Choose
		$x^*(\lambda) \in X$ with
		$\phi_\lambda(x^*(\lambda)) \le \phi^*_\lambda + 1/\lambda$. For
		$\eps > 0$, pick $\tilde x \in X$ with
		$\|Z(\tilde x)_+\|_2^2 \le r^2 + \eps$. Then
		$\phi^*_\lambda \le \phi_\lambda(\tilde x) = -U(\tilde x) + (\lambda/2)\|Z(\tilde x)_+\|_2^2$,
		so
		\[
		-U(x^*(\lambda)) + \tfrac{\lambda}{2}\|Z(x^*(\lambda))_+\|_2^2
		\;\le\; -U(\tilde x) + \tfrac{\lambda}{2}(r^2 + \eps) + \tfrac{1}{\lambda},
		\]
		and after rearrangement,
		\begin{equation}\label{eq:rate}
			\|Z(x^*(\lambda))_+\|_2^2 \;\le\; r^2 + \eps + \tfrac{2}{\lambda}\bigl(U_{\sup} - U(\tilde x)\bigr) + \tfrac{2}{\lambda^2}.
		\end{equation}
		Sending $\lambda \to \infty$ then $\eps \to 0$ yields
		$\limsup_\lambda \|Z(x^*(\lambda))_+\|_2 \le r$. The reverse
		inequality follows from the definition of $r$. By
		\eqref{eq:rate}, the family $\{Z(x^*(\lambda))_+\}$ is bounded
		in $\IR^L$ for $\lambda$ bounded away from zero, so by
		Bolzano--Weierstrass it admits a subsequential limit
		$v^*_\infty$, which then satisfies $\|v^*_\infty\|_2 = r$.
	\end{proof}
	
	\begin{remark}[Norm choice and modeling interpretation]\label{rem:norms}
		Replacing $(\lambda/2)\|v\|_2^2$ in~\eqref{eq:relax} by
		$\lambda\|v\|_1$ (resp.\ $\lambda\|v\|_\infty$) yields an
		analogue of Theorem~\ref{thm:nonex} in which $v^*_\infty$
		minimizes the corresponding norm: under $\ell_1$, $v^*_\infty$
		minimizes the total intervention and is sparse in a few goods;
		under $\ell_\infty$, it minimizes the maximum per-good shortage
		and represents a balanced rationing; under $\ell_2$, it is the
		Euclidean projection. The choice is a modeling decision,
		fixing the operational meaning of ``closest equilibrium-admitting
		economy'' within a given application.
	\end{remark}
	
	\begin{remark}[Convergence rate and conditioning]\label{rem:rate}
		The bound \eqref{eq:rate} delivers convergence at rate
		$\bigO(1/\lambda)$, with a constant proportional to the spread
		of $U$ along the trajectory. In applications involving
		subsistence constraints this constant scales poorly,
		necessitating large $\lambda$. In the example below we take
		$\lambda$ up to $10^7$ and obtain stable iterates with a
		quasi-Newton solver (L-BFGS-B) under warm starts along the
		$\lambda$ sweep.
	\end{remark}

	\section{An illustration: the Shapley--Shubik $2 \times 2$ economy}\label{sec:numerics}
	
	We consider a parametrization of the $2 \times 2$ economy of
	\citet{ShapShub77}, generalized by \citet{BergSY09} and reproduced
	in \citet[Ex.~17.F.1]{MWG95}. The two agents have quasi-linear
	utilities
	\[
	u_1(x_{11}, x_{21}) = x_{11} - \tfrac{1}{8}x_{21}^{-8},
	\qquad u_2(x_{12}, x_{22}) = x_{22} - \tfrac{1}{8}x_{12}^{-8},
	\]
	with survival sets $X_i = [0.3, 3]^2$. With endowments
	$e^1 = (2, r)$ and $e^2 = (r, 2)$ for $r = 2^{8/9} - 2^{1/9}$, the
	economy admits three Walrasian equilibria at
	$p^* \in \{1/2, 1, 2\}$ (each verified to machine precision), as
	constructed in \citet{BergSY09}: strong income effects act in the
	opposite direction from substitution effects, producing a
	non-monotone aggregate excess-demand function. Classical
	t\^atonnement and the augmented-Walrasian iteration of
	\citet{DerJW19} both converge in this baseline regime.
	
	\paragraph{Stressed variant.} We shrink endowments to $e^1 = (0.50, 0.05)$
	and $e^2 = (0.05, 0.50)$. The aggregate endowment $(0.55, 0.55)$ is
	strictly below twice the survival floor $(0.6, 0.6)$, so the
	feasibility set $\{x \in [0.3, 3]^4 : Z(x) \le 0\}$ is empty: no
	Walrasian equilibrium exists. The infeasibility floor admits a closed
	form. The minimum of $\|Z(x)_+\|_2^2$ over $X$ is attained at
	$x^* = (0.3, 0.3, 0.3, 0.3)$, giving
	\[
	r \;=\; \sqrt{(0.6 - 0.55)^2 + (0.6 - 0.55)^2} \;=\; \sqrt{2} \cdot 0.05 \;\approx\; 0.0707.
	\]
	
	We solve relaxation~\eqref{eq:relax} via L-BFGS-B from a uniform
	interior initial allocation, sweeping $\lambda \in [10^0, 10^7]$
	logarithmically with warm starts. Figure~\ref{fig:slack} reports the
	residual $v^*(\lambda)$ component-wise (left) and its $\ell_2$-norm
	(right). The norm decreases monotonically and converges to
	$\|v^*_\infty\|_2 \approx 0.0707$, matching the analytical value
	predicted by Theorem~\ref{thm:nonex}. Each per-good shortage equals
	$0.05$, the per-coordinate gap of aggregate endowment relative to
	survival.
	
	\begin{figure}[htbp]
		\centering
		\includegraphics[width=\linewidth]{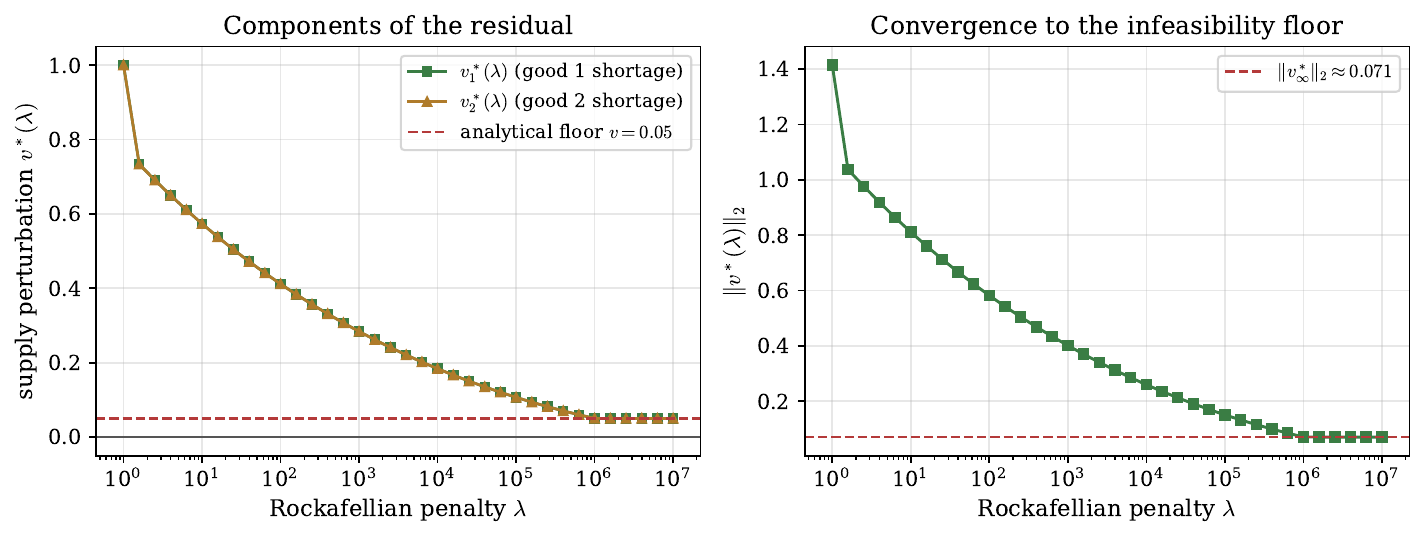}
		\caption{\label{fig:slack}Stressed Shapley--Shubik economy. Left:
			components of $v^*(\lambda) = Z(x^*(\lambda))_+$. Right:
			$\|v^*(\lambda)\|_2$ converges to the analytical infeasibility
			floor $\sqrt{2} \times 0.05 \approx 0.0707$ predicted by
			Theorem~\ref{thm:nonex}.}
	\end{figure}
	
	The residual $v^*_\infty$ admits a direct reading. Under
	$\ell_2$, the economy lies $0.0707$ units away in Euclidean norm
	from the nearest equilibrium-admitting one. Under $\ell_1$, the
	corresponding minimum is $0.10$ units in total; under
	$\ell_\infty$, it is $0.05$ in the worst-case good. By contrast,
	classical t\^atonnement diverges in this regime, and the
	augmented-Walrasian iteration of \citet{DerJW19} stagnates
	without producing a residual that admits an economic
	reading.\footnote{Both behaviors are reproduced in the
	accompanying repository.}

	\section{Concluding remarks}\label{sec:concl}
	
	We have proposed a Rockafellian relaxation of the Walrasian
	equilibrium problem that remains well posed when no equilibrium
	exists, and we have shown that its residual converges, as the
	penalty parameter grows, to a vector $v^*_\infty$ of minimum norm
	in the feasible range of the excess-demand map. We interpret
	$\|v^*_\infty\|$ as the minimum-norm slack required to restore
	feasibility, and we have argued that the choice of norm in the
	penalty fixes the policy reading.
	
	Two extensions follow naturally. In incomplete-markets economies
	\citep{DerJWXR}, $v^*_\infty$ quantifies a minimum-norm
	no-arbitrage perturbation; in two-stage stochastic economies, the
	relaxation decomposes by scenario, and the stability and rate
	properties of \citet{DerR25} extend directly to that decomposed
	setting. We note one limitation of our approach: the closed-form
	profile-out of $v$ in~\eqref{eq:relax} relies on the separability
	of $\|v\|_p^p$ in $v$ alone, so for general perturbations one
	would carry $v$ explicitly and solve the Rockafellian jointly in
	$(x, v)$.
	
	\paragraph{Competing interests.} The author declares no competing
	interests.
	
	\section*{Code and data availability}
	A Python implementation reproducing every number and figure in this
	note is available at
	\url{https://github.com/jderide/rockafellian-equilibria}.
	
	\bibliographystyle{abbrvnat}

\end{document}